\documentclass[12pt,a4paper]{amsart}
\usepackage{amsfonts}
\usepackage{amsthm}
\usepackage{amsmath}
\usepackage{amscd}
\usepackage[latin2]{inputenc}
\usepackage{t1enc}
\usepackage[mathscr]{eucal}
\usepackage{indentfirst}
\usepackage{graphicx}
\usepackage{graphics}
\usepackage{subcaption}
\numberwithin{equation}{section}

\theoremstyle{plain}
\newtheorem{Th}{Theorem}[section]
\newtheorem{Lemma}[Th]{Lemma}
\newtheorem{Cor}[Th]{Corollary}

\newtheorem{Pro}[Th]{Proposition}

 \theoremstyle{definition}
\newtheorem{Def}[Th]{Definition}

\newtheorem{Rem}[Th]{Remark}
\newtheorem{?}[Th]{Problem}

\newcommand{\oI}{I}

\DeclareMathSizes{5}{5}{5}{5}

\usepackage{float}
\usepackage{tikz}
\usetikzlibrary{arrows}
\tikzstyle{n}=[circle,draw,thick,scale=0.5]
\tikzstyle{nt}=[circle,draw,thick,fill,scale=0.5]
\tikzstyle{nc}=[circle,draw,thick,fill]

\begin{document}

\title{On trees with real rooted independence polynomial}
\thanks{The author was partially supported by the MTA R\'enyi Institute Lend\"ulet Limits of Structures Research Group.}

\author{Ferenc Bencs} 
 \address{Central European University, Department of Mathematics
 \\ H-1051 Budapest
 \\ Zrinyi u. 14, Third Floor \\ Hungary \& Alfr\'ed R\'enyi Institute of Mathematics\\ H-1053 Budapest\\ Re\'altanoda u. 13-15.} 
 \email{ferenc.bencs@gmail.com}



\begin{abstract} 
The independence polynomial of a graph $G$ is \[I(G,x)=\sum\limits_{k\ge 0}i_k(G)x^k,\] where  $i_k(G)$ denotes the number of independent sets of $G$ of size $k$ (note that $i_0(G)=1$).  In this paper we show a new method to prove real-rootedness of the independence polynomials of certain families of trees. 

In particular we will give a new proof of the  real-rootedness of the independence polynomials of centipedes (Zhu's theorem),  caterpillars (Wang and Zhu's theorem), and we will prove a conjecture of Galvin and Hilyard about the real-rootedness of the independence polynomial of the so-called Fibonacci trees. 
\end{abstract}

\maketitle

\section{Introduction} 
The independence polynomial of a graph $G$ is
\small
\[
	I(G,x)=\sum\limits_{k\ge 0}i_k(G)x^k,
\]
where  $i_k(G)$ denotes the number of independent sets of $G$ of size $k$ (note that $i_0(G)=1$). In this paper we study the independence polynomials of trees. For trees, it is a well known conjecture that the sequence $(i_k(T))_{k\ge 0}$ is  unimodal.

Recall that a sequence $(b_k)_{k= 0}^{n}$ is unimodal (\cite{Stanley1989}), if there exists an index $k$, such that 
\[
  b_0\le b_1\le\dots\le b_{k-1}\le b_k\ge b_{k+1}\ge \dots \ge b_{n}.
\]
A stronger property for positive sequences is the so called log-concavity: for any $0< i< n$ we have $b_i^2\ge b_{i-1}b_{i+1}$. An even stronger property is the real-rootedness of the polynomial $p(x)=\sum_{i=0}^n b_ix^i$ (any complex zero of the polynomial is real). 
This prompted many mathematicians to study trees with real-rooted independence polynomials. In this paper we show a general method to construct such trees or prove real-rootedness. 

In particular we will give a new proof for  real-rootedness of the independence polynomials of certain  families of trees, which  includes centipedes (Zhu's theorem, see \cite{Zhu2007}),  caterpillars (Wang and Zhu's theorem, see \cite{Wang2011}), and we will prove a conjecture of Galvin and Hilyard about the real-rootedness of the independence polynomial of the Fibonacci trees (Conj. 6.1. of \cite{Galvin2017}). 
%centipedes (Thm. 1 of \cite{Zhu2007}),  caterpillars (Prop. 2.1 of \cite{Wang2011}), and we will prove a conjecture of Galvin and Hilyard about Fibonacci-trees (Conj. 6.1. of \cite{Galvin2017}). 
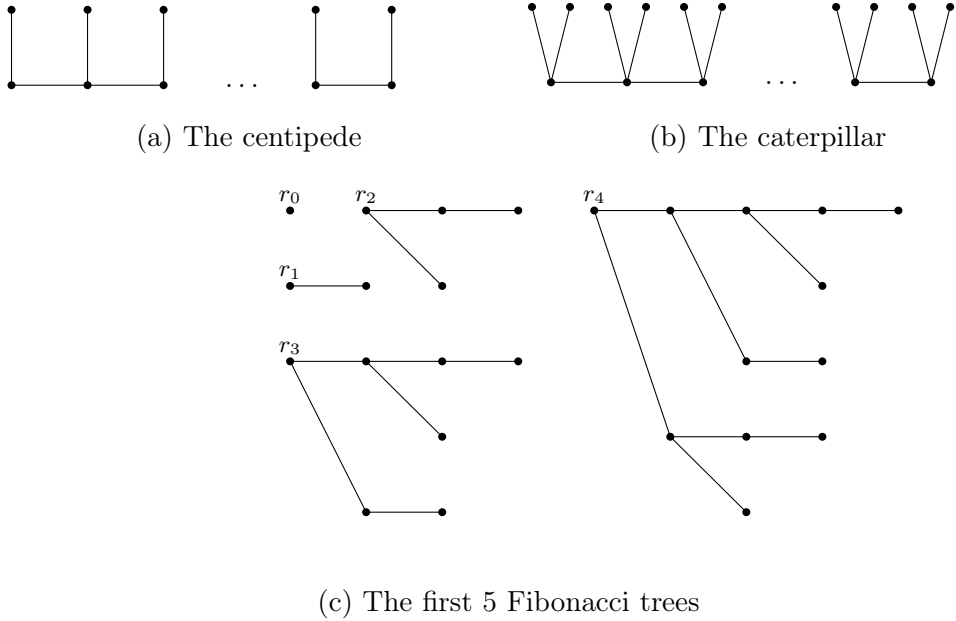
\begin{figure}
\begin{subfigure}[b]{0.5\linewidth}
      \begin{tikzpicture}[line cap=round,line join=round,>=triangle 45,x=1.0cm,y=1.0cm]
      \clip(0.71,0.73) rectangle (6.68,2.11);
      \draw (1,1)-- (1,2);
      \draw (1,1)-- (2,1);
      \draw (2,1)-- (3,1);
      \draw (2,1)-- (2,2);
      \draw (3,1)-- (3,2);
      \draw (5,1)-- (6,1);
      \draw (6,1)-- (6,2);
      \draw (5,1)-- (5,2);
      \draw (4,1) node {$$ \dots $$};
      \begin{scriptsize}
      \fill  (1,1) circle (1.5pt);
      \fill  (2,1) circle (1.5pt);
      \fill  (3,1) circle (1.5pt);
      \fill  (5,1) circle (1.5pt);
      \fill  (6,1) circle (1.5pt);
      \fill  (1,2) circle (1.5pt);
      \fill  (2,2) circle (1.5pt);
      \fill  (3,2) circle (1.5pt);
      \fill  (5,2) circle (1.5pt);
      \fill  (6,2) circle (1.5pt);
      \end{scriptsize}
      \end{tikzpicture}
      \caption{The centipede}
      \label{fig:centipede}
\end{subfigure}%
\begin{subfigure}[b]{0.5\linewidth}
      \begin{tikzpicture}[line cap=round,line join=round,>=triangle 45,x=1.0cm,y=1.0cm]
      \clip(0.44,0.69) rectangle (6.8,2.23);
      \draw (1,1)-- (0.75,2);
      \draw (1,1)-- (2,1);
      \draw (2,1)-- (3,1);
      \draw (2,1)-- (1.75,2);
      \draw (3,1)-- (2.75,2);
      \draw (5,1)-- (6,1);
      \draw (6,1)-- (5.75,2);
      \draw (5,1)-- (4.75,2);
      \draw (4,1) node {$$ \dots $$};
      \draw (1,1)-- (1.25,2);
      \draw (2,1)-- (2.25,2);
      \draw (3,1)-- (3.25,2);
      \draw (5,1)-- (5.25,2);
      \draw (6,1)-- (6.25,2);
      \begin{scriptsize}
      \fill  (1,1) circle (1.5pt);
      \fill  (2,1) circle (1.5pt);
      \fill  (3,1) circle (1.5pt);
      \fill  (5,1) circle (1.5pt);
      \fill  (6,1) circle (1.5pt);
      \fill  (0.75,2) circle (1.5pt);
      \fill  (1.75,2) circle (1.5pt);
      \fill  (2.75,2) circle (1.5pt);
      \fill  (4.75,2) circle (1.5pt);
      \fill  (5.75,2) circle (1.5pt);
      \fill  (1.25,2) circle (1.5pt);
      \fill  (2.25,2) circle (1.5pt);
      \fill  (3.25,2) circle (1.5pt);
      \fill  (5.25,2) circle (1.5pt);
      \fill  (6.25,2) circle (1.5pt);
      \end{scriptsize}
      \end{tikzpicture}
      \caption{The caterpillar}
      \label{fig:caterpillar}
\end{subfigure}

\begin{subfigure}[b]{0.5\linewidth}
      \begin{tikzpicture}[line cap=round,line join=round,>=triangle 45,x=1.0cm,y=1.0cm]
      \clip(-1.54,-1.78) rectangle (7.95,3.69);
      \draw (5,3)-- (6,3);
      \draw (6,3)-- (7,3);
      \draw (5,3)-- (6,2);
      \draw (5,1)-- (6,1);
      \draw (4,3)-- (5,3);
      \draw (4,3)-- (5,1);
      \draw (4,0)-- (5,0);
      \draw (5,0)-- (6,0);
      \draw (4,0)-- (5,-1);
      \draw (3,3)-- (4,3);
      \draw (3,3)-- (4,0);
      \draw (0,1)-- (1,1);
      \draw (1,1)-- (2,1);
      \draw (0,1)-- (1,0);
      \draw (0,-1)-- (1,-1);
      \draw (-1,1)-- (0,1);
      \draw (-1,1)-- (0,-1);
      \draw (0,3)-- (1,3);
      \draw (1,3)-- (2,3);
      \draw (0,3)-- (1,2);
      \draw (-1,2)-- (0,2);
      \begin{scriptsize}
      \fill  (6,3) circle (1.5pt);
      \fill  (7,3) circle (1.5pt);
      \fill  (6,2) circle (1.5pt);
      \fill  (5,3) circle (1.5pt);
      \fill  (5,1) circle (1.5pt);
      \fill  (6,1) circle (1.5pt);
      \fill  (4,3) circle (1.5pt);
      \fill  (5,0) circle (1.5pt);
      \fill  (6,0) circle (1.5pt);
      \fill  (5,-1) circle (1.5pt);
      \fill  (4,0) circle (1.5pt);
      \fill  (3,3) circle (1.5pt);
      \draw (3,3.18) node {$r_4$};
      \fill  (1,1) circle (1.5pt);
      \fill  (2,1) circle (1.5pt);
      \fill  (1,0) circle (1.5pt);
      \fill  (0,1) circle (1.5pt);
      \fill  (0,-1) circle (1.5pt);
      \fill  (1,-1) circle (1.5pt);
      \fill  (-1,1) circle (1.5pt);
      \draw (-1,1.18) node {$r_3$};
      \fill  (1,3) circle (1.5pt);
      \fill  (2,3) circle (1.5pt);
      \fill  (1,2) circle (1.5pt);
      \fill  (0,3) circle (1.5pt);
      \draw (0,3.18) node {$r_2$};
      \fill  (-1,2) circle (1.5pt);
      \draw (-1,2.18) node {$r_1$};
      \fill  (0,2) circle (1.5pt);
      \fill  (-1,3) circle (1.5pt);
      \draw (-1,3.18) node {$r_0$};
      \end{scriptsize}
      \end{tikzpicture}
      \caption{The first 5 Fibonacci trees}
      \label{fig:fibonacci}
\end{subfigure}
\caption{Some families of trees}
\end{figure}

Recall that the \textsl{$n$-centipede $W_n$} is a graph (Fig.~\ref{fig:centipede}), such that we take a path on $n$ vertices and we hang 1 pendant edge from each vertex of it.
  Similarly the \textsl{$n$-caterpillar} $H_n$ is the graph (Fig.~\ref{fig:caterpillar}) obtained by taking a path on $n$ vertices and by hanging 2 pendant edges from each vertex of it. The \textsl{Fibonacci trees} were defined by Wagner \cite{Wagner2007} as follows (Fig.~\ref{fig:fibonacci}): let $F_0=K_1$ and $F_1=K_2$ with roots $r_0\in V(F_0)$ and $r_1\in V(F_1)$. Then for $n\ge 2$ the $n$th Fibonacci tree $F_n$ is obtained from  the disjoint union of $F_{n-1}$, $F_{n-2}$ and a new vertex, labeled by $r_n$ and connecting $r_n$ to the roots of $F_{n-1}$ and $F_{n-2}$. Define $r_n$ as the root of $F_n$.

\subsection{Methods and motivations}

To motivate our method we will use certain results from the theory of matching polynomials. Recall that the matching polynomial of a graph $G$ is defined as:
\[
	\mu(G,x)=\sum\limits_{k=0}(-1)^km_k(G)x^{n-2k},
\]
where $m_k(G)$ is the number of matchings with $k$ edges (note that $m_0(G)=1$).
One of the best known theorems about matching polynomials is that for any finite graph $G$ and $u\in V(G)$ there exists a rooted tree $(T,r)$, such that
\begin{eqnarray}
  \frac{\mu(G-u,x)}{\mu(G,x)}=\frac{\mu(T-r,x)}{\mu(T,x)} 
\end{eqnarray}
A well-known construction for $T$ is the path-tree \cite{Godsil1993} (a.k.a. Godsil tree), which is the tree on paths of $G$ starting from $u$, and the edges are the strict inclusions. (For an example see Figure~\ref{fig:pt}.)
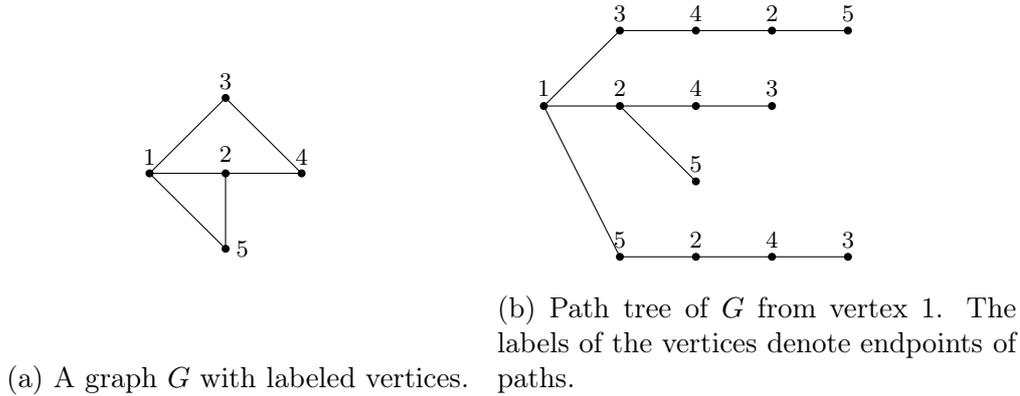
\begin{figure}
\centering
  \begin{subfigure}[b]{0.5\linewidth}
    \centering
    \begin{tikzpicture}[line cap=round,line join=round,>=triangle 45,x=1.0cm,y=1.0cm]
      \clip(0.72,0.7) rectangle (3.48,3.64);
      \draw (1,2)-- (2,2);
      \draw (1,2)-- (2,1);
      \draw (2,2)-- (2,1);
      \draw (1,2)-- (2,3);
      \draw (2,2)-- (3,2);
      \draw (2,3)-- (3,2);
      \begin{scriptsize}
      \fill  (1,2) circle (1.5pt);
      \draw (1,2.22) node {$1$};
      \fill  (2,2) circle (1.5pt);
      \draw (2,2.26) node {$2$};
      \fill  (2,3) circle (1.5pt);
      \draw (2,3.22) node {$3$};
      \fill  (2,1) circle (1.5pt);
      \draw (2.22,1) node {$5$};
      \fill  (3,2) circle (1.5pt);
      \draw (3,2.22) node {$4$};
      \end{scriptsize}
    \end{tikzpicture}
    \vspace{1cm}
    \caption{A graph $G$ with labeled vertices.} 
  \end{subfigure}%
  \begin{subfigure}[b]{0.5\linewidth}
    \begin{tikzpicture}[line cap=round,line join=round,>=triangle 45,x=1.0cm,y=1.0cm]
    \clip(0.39,-0.28) rectangle (6.31,3.61);
    \draw (1,2)-- (2,3);
    \draw (1,2)-- (2,2);
    \draw (1,2)-- (2,0);
    \draw (2,3)-- (3,3);
    \draw (3,3)-- (4,3);
    \draw (4,3)-- (5,3);
    \draw (2,2)-- (3,2);
    \draw (3,2)-- (4,2);
    \draw (2,2)-- (3,1);
    \draw (2,0)-- (3,0);
    \draw (3,0)-- (4,0);
    \draw (4,0)-- (5,0);
    \begin{scriptsize}
    \fill  (1,2) circle (1.5pt);
    \draw (1,2.23) node {$1$};
    \fill  (2,3) circle (1.5pt);
    \draw (2,3.23) node {$3$};
    \fill  (2,2) circle (1.5pt);
    \draw (2,2.23) node {$2$};
    \fill  (2,0) circle (1.5pt);
    \draw (2,0.23) node {$5$};
    \fill  (3,1) circle (1.5pt);
    \draw (3,1.23) node {$5$};
    \fill  (3,2) circle (1.5pt);
    \draw (3,2.23) node {$4$};
    \fill  (4,2) circle (1.5pt);
    \draw (4,2.23) node {$3$};
    \fill  (3,3) circle (1.5pt);
    \draw (3,3.23) node {$4$};
    \fill  (4,3) circle (1.5pt);
    \draw (4,3.23) node {$2$};
    \fill  (5,3) circle (1.5pt);
    \draw (5,3.23) node {$5$};
    \fill  (3,0) circle (1.5pt);
    \draw (3,0.23) node {$2$};
    \fill  (4,0) circle (1.5pt);
    \draw (4,0.23) node {$4$};
    \fill  (5,0) circle (1.5pt);
    \draw (5,0.23) node {$3$};
    \end{scriptsize}
    \end{tikzpicture}
    \caption{Path tree of $G$ from vertex 1. The labels of the vertices denote endpoints of paths.}
  \end{subfigure}
  \caption{A graph with its path tree.}
  \label{fig:pt}
\end{figure}

In this paper we will prove an ''independence version'' of this theorem through a quite similar construction. More precisely, we will show that  there exists a rooted tree $(T',r)$, such that
\[
  \frac{\oI(G-u,x)}{\oI(G,x)}=\frac{\oI(T'-r,x)}{\oI(T',x)}.
\]
We will call the constructed tree a stable-path tree. This construction already appeared in in the work of Scott and Sokal (see \cite{Scott2005}) and variant of this construction in the work of Weitz (see \cite{Weitz2006}).
We will see that the key property of a stable-path tree is that its independence polynomial is a product of independence polynomials of some induced subgraphs of $G$. 

Chudnovsky and Seymour showed that the independence polynomial of any claw-free graph is real-rooted (see \cite{sey}). Since any induced subgraph of a claw-free graph is also claw-free, this enables us to conclude that any stable-path tree of a claw-free graph has real-rooted independence polynomial. In section ~\ref{sec:ap}  we will construct claw-free graphs such that their stable-path trees will be $n$-centipedes, $n$-caterpillars and Fibonacci trees. In the same section we will give further applications of this method.

\subsection{Notation}
We denote the vertex set and edge set of a graph $G$ by $V(G)$ and $E(G)$, respectively.
Let $N_G(u)$ denote the set of neighbours of the vertex $u$ and $d(u)$ the degree of the vertex $u$.
Let $N_G[u]=N_G(u)\cup\{u\}$ denote the closed neighbourhood of the vertex $u$. If it is clear from the context, then we will write $N(u)$ and $N[u]$ instead of $N_G(u)$ and $N_G[u]$.
Let $G-v$ denote the graph obtained from $G$ by deleting the vertex $v$.
If $S\subseteq V(G)$, then $G[S]$ denotes the induced subgraph of $G$ on the vertex set $S$, and $G-S$ denotes $G[V(G)-S]$.

\subsection{This paper is organized as follows:} in the next section we will define  stable-path trees of graphs, and we will prove some properties of it. In the last section we will prove real-rootedness of independence polynomials of certain graphs.

\medskip
\section{Tree of stable paths}\label{sec:main}

In this section we will give two variants of the definition of the stable-path tree, where the first one is a special case of the latter one. For the applications it is enough to get familiar with the first definition.
But first let us recall the following properties of the independence polynomial, which we will use intensively in the proofs. For proof see \cite{Levit2005}.
\begin{Lemma}
 Let $G$ be a graph with connected components $G_1,\dots,G_k$, and let $u\in V(G)$ be a fixed vertex. Then
 \begin{gather*}
  I(G,x)=I(G-u,x)+xI(G-N_G[u],x)\\
  I(G,x)=\prod_{i=1}^kI(G_i,x)
 \end{gather*}
\end{Lemma}

\begin{Def}[Tree of stable paths]
 Let $G$ be a graph, where we have a total ordering $\prec$ on $V(G)$ and let  $u\in V(G)$ fixed. Then we define a tree $(T^<_{G,u},\bar{u})$ as follows. Let us denote by $N(u)=\{u_1\prec \dots \prec u_d\}$, and let 
 \begin{gather*}
  G^i=G[V(G)\setminus\{u,u_1,v_2,\dots,u_{i-1}\}]\\
  (T^i,r^i)=(T^<_{G^i,u_i},\bar{u_i}),
  \end{gather*}
  where we take the induced ordering of the vertices on $V(G^i)$ for $1\le i\le d$. Consider the disjoint unions of $T^i$ with roots $r^i$ and a new vertex with label $\bar{u}$, and add edges $(\bar{u},r^i)$ for $1\le i\le d$. In this way we gain a tree $T^<_{G,u}$ and let $\bar{u}$ be the root of this tree. See an example in Fig~\ref{fig:spt}.
\end{Def}
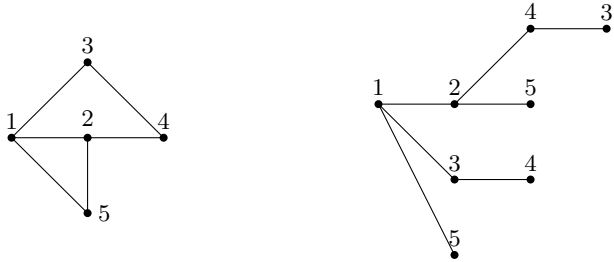
\begin{figure}[h!]
\centering
\begin{subfigure}[b]{0.5\linewidth}
\centering
\begin{tikzpicture}[line cap=round,line join=round,>=triangle 45,x=1.0cm,y=1.0cm]
\clip(0.72,0.7) rectangle (3.48,3.64);
\draw (1,2)-- (2,2);
\draw (1,2)-- (2,1);
\draw (2,2)-- (2,1);
\draw (1,2)-- (2,3);
\draw (2,2)-- (3,2);
\draw (2,3)-- (3,2);
\begin{scriptsize}
\fill  (1,2) circle (1.5pt);
\draw (1,2.22) node {$1$};
\fill  (2,2) circle (1.5pt);
\draw (2,2.26) node {$2$};
\fill  (2,3) circle (1.5pt);
\draw (2,3.22) node {$3$};
\fill  (2,1) circle (1.5pt);
\draw (2.22,1) node {$5$};
\fill  (3,2) circle (1.5pt);
\draw (3,2.22) node {$4$};
\end{scriptsize}
\end{tikzpicture}
\vspace{1cm}
 \caption{A graph $G$ with labeled vertices.} 
\end{subfigure}%
\begin{subfigure}[b]{0.5\linewidth}
\begin{tikzpicture}[line cap=round,line join=round,>=triangle 45,x=1.0cm,y=1.0cm]
\clip(0.75,0.73) rectangle (4.48,4.6);
\draw (1,3)-- (2,3);
\draw (1,3)-- (2,2);
\draw (1,3)-- (2,1);
\draw (2,2)-- (3,2);
\draw (2,3)-- (3,4);
\draw (3,4)-- (4,4);
\draw (2,3)-- (3,3);
\begin{scriptsize}
\fill  (1,3) circle (1.5pt);
\draw (1,3.22) node {$1$};
\fill  (2,3) circle (1.5pt);
\draw (2,3.22) node {$2$};
\fill  (2,2) circle (1.5pt);
\draw (2,2.22) node {$3$};
\fill  (2,1) circle (1.5pt);
\draw (2,1.22) node {$5$};
\fill  (3,4) circle (1.5pt);
\draw (3,4.23) node {$4$};
\fill  (4,4) circle (1.5pt);
\draw (4,4.22) node {$3$};
\fill  (3,3) circle (1.5pt);
\draw (3,3.22) node {$5$};
\fill  (3,2) circle (1.5pt);
\draw (3,2.22) node {$4$};
\end{scriptsize}
\end{tikzpicture}

 \caption{The graph $T^<_{G,1}$. The labels of the vertices denote endpoints of stable-paths.}
\end{subfigure}

 \caption{A graph with its stable-path tree. The ordering on the vertices of $G$ is induced by its labeling}
 \label{fig:spt}
\end{figure}

\begin{Th}\label{th:fa}
 Let $G$ be a graph, $u\in V(G)$. Then for $T=T^<_{G,u}$ we have that
 \[
  \frac{\oI(G-u,x)}{\oI(G,x)}=\frac{\oI(T-\overline{u},x)}{\oI(T,x)},
 \]
\end{Th}
\begin{proof}
 We will prove the statement by induction on the number of vertices of $G$. If $G$ has exactly one vertex, then $T^<_{G,u}$ is constructed to be a graph with one vertex. 
 
 Let $N(u)=\{u_1\prec \dots \prec u_d\}$, and then let $G^i=G[V(G)\setminus\{u,u_1,v_2,\dots,u_{i-1}\}]$ and $(T^i,r^i)=(T_{G^i,u_i},\bar{u_i})$ for $1\le i \le d$ as in the definition. Then 
 \begin{gather*}
  \frac{\oI(G,x)}{\oI(G-u,x)}=\frac{\oI(G-u,x)+x\oI(G-N[u],x)}{\oI(G-u,x)}=1+\frac{x\oI(G-N[u],x)}{\oI(G-u,x)}=\\
  =1+x\frac{\oI(G-u-u_1,x)\oI(G-u-\{u_1,u_2\},x)\dots \oI(G-u-\{u_1,\dots,u_k\},x)}{\oI(G-u,x)\oI(G-u-u_1)\dots \oI(G-u-\{u_1,\dots,u_{k-1}\})}=\\
  =1+x\frac{\oI(G^1-
  u_1,x)}{\oI(G^1,x)}\frac{\oI(G^2-u_2,x)}{\oI(G^2,x)}\dots\frac{\oI(G^d-u_d,x)}{\oI(G^d,x)}=\\
  =1+x\frac{\oI(T^1-r^1,x)}{\oI(T^1,x)}\frac{\oI(T^2-r^2,x)}{\oI(T^2,x)}\dots\frac{\oI(T^d-r^d,x)}{\oI(T^d,x)}=\\  
  =\frac{\oI(T-r,x)+x\oI(T-N[r],x)}{\oI(T-r,x)}=\frac{\oI(T,x)}{\oI(T-r,x)}.
  \end{gather*}
% where $T$ is a tree such that is we take disjoint union of a new vertex $r$ and $T^i$'s for $1\le i\le d$ and then connecting $r$ with the roots of $T^i$ for $1\le i\le d$. Clearly $(T,r)$ is isomorphic to $(T_{G,u},\bar{u})$.
\end{proof}

We would like to remark that in all applications it will be enough to use this definition, however, for the completeness we will give a a more general form.

The following construction already appeared in the work of Scott and Sokal (see \cite{Scott2005}), where they called the this tree as pruned SAW-tree. 

  \begin{Def}[Tree of $\sigma$-stable paths]
  Let  $\mathcal{P}_u$ be the set of paths from $u$ in  $G$, and let \[A_{G,u}=\{(P,e)\in \mathcal{P}_u\times E(G)~|~P=(v_0,\dots,v_k), v_k\in e\}.\] A function $\sigma:A_{G,u}\to \mathbb{R}$ is called deep decision if it satisfies that whenever $(P,e),(P,f)\in A_{G,u}$ and $\sigma(P,e)=\sigma(P,f)$, then $e=f$. Then a path $P=(v_0,v_1, \dots, v_k)$ from $u$ is $\sigma$-stable, if whenever $(v_i,v_j)\in E(G)$ and $i+1<j$, then $\sigma(P',(v_i,v_{i+1}))<\sigma(P',(v_i,v_j))$, where $P'=(v_0,\dots, v_i)$ is a subpath. 
 If the path $P=(u,v_1,\dots,v_k)$ is stable with respect to $\sigma$, then $P'=(u,v_1,\dots,v_{k-1})$ is also stable with respect to $\sigma$.
 
 Let $T^\sigma_{G,u}$ be a tree, whose vertices are the $\sigma$-stable paths from $u$, and the edges correspond to the strict inclusion. In that tree the path $(u)$ (with length 0) appears, which we will denote by $\overline{u}$.
\end{Def}

To see the relation between the two definitions, let us assume, that $G$ has a total ordering on its vertices, so we may assume, that $(V(G),\prec)=(\{1,\dots,n\},<)$. Then for a $(P,e)\in A_{G,u}$, such that $P=(u,v_1,\dots,v_k)$ and $e=
(v_k,v_{k+1})$ let $\sigma(P,e)=v_{k+1}$. Then it is easy to check that $T^\sigma_{G,u}=T^<_{G,u}$. Indeed the second definition is a generalization of the first one.

For the completeness we will prove Theorem~\ref{th:fa} also for the generalized $\sigma$-stable-path tree.
\begin{Th}\label{th:fa_gen}
 Let $G$ be a graph, $u\in V(G)$ and let $\sigma:A_{G,u}\to \mathbb{R}$ be a deep decision. Then for $T=T^{\sigma}_{G,u}$ we have that
 \[
  \frac{\oI(G-u,x)}{\oI(G,x)}=\frac{\oI(T-\overline{u},x)}{\oI(T,x)},
 \]
\end{Th}

\begin{proof}
  We will prove the statement by induction on the number of vertices. If $G$ has exactly one vertex, then $T^\sigma(G,u)$ is constructed to be a graph with one vertex.
  
  Furthermore we may assume that $G$ is connected, since if $G_1, \dots, G_k$ are the connected components of $G$, where $u\in V(G_1)$, then by using the multiplicity of the independence polynomial, we have: 
 \[
  \frac{\oI(G-u,x)}{\oI(G,x)}=\frac{\oI(G_1-u,x)\oI(G_2,x)\dots \oI(G_k,x)}{\oI(G_1,x)\oI(G_2,x)\dots \oI(G_k,x)}=\frac{\oI(G_1-u,x)}{\oI(G_1,x)}.
 \] and by $A_{G,u}=A_{G_1,u}$ we have that $T^\sigma(G_1,u)=T^\sigma(G,u)$, which is the appropriate tree.
 
 For the rest assume that $G$ is connected. Then let $N(u)=\{u_1,\dots,u_d\}$ in such a way, such that $\sigma(\overline{u},(u,u_i))<\sigma(\overline{u},(u,u_j))$, whenever $1\le i<j\le d$. Then for any $1\le i\le d$ and for any $(P,e)\in A_{G-\{u,u_1,\dots,u_{i-1}\},u_i}$ let $\sigma_i$ be defined as follows (where $P=(u_i,v_1,\dots,v_k)$):
 \[
  \sigma_i(P,e)=\sigma((u,u_i,v_1,\dots v_k), e).
 \]
 Then
 \begin{gather*}
  \frac{\oI(G,x)}{\oI(G-u,x)}=\frac{\oI(G-u,x)+x\oI(G-N[u],x)}{\oI(G-u,x)}=1+\frac{x\oI(G-N[u],x)}{\oI(G-u,x)}=\\
  =1+x\frac{\oI(G-u-u_1,x)\oI(G-u-\{u_1,u_2\},x)\dots \oI(G-u-\{u_1,\dots,u_d\},x)}{\oI(G-u,x)\oI(G-u-u_1)\dots \oI(G-u-\{u_1,\dots,u_{d-1}\})}=\\
  =1+x\frac{\oI(G-u-
  u_1,x)}{\oI(G-u,x)}\frac{\oI(G-u-\{u_1,u_2\},x)}{\oI(G-u-u_1)}\dots\frac{\oI(G-u-\{u_1,\dots,u_d\},x)}{\oI(G-u-\{u_1,\dots,u_{d-1}\})}=\\
  =1+x\frac{\oI(T^{\sigma_1}_{G-u,u_1}-\overline{u_1},x)}{\oI(T^{\sigma_1}_{G-u,u_1},x)}\frac{\oI(T^{\sigma_2}_{G-u-u_1,u_2}-\overline{u_2},x)}{\oI(T^{\sigma_2}_{G-u-u_1,u_2},x)}\dots\frac{\oI(T^{\sigma_d}_{G-u-\{u_1\dots u_{d-1}\},u_d}-\overline{u_d},x)}{\oI(T^{\sigma_d}_{G-u-\{u_1\dots u_{d-1}\},d_k},x)}=\\
  =\frac{\oI(T,x)}{\oI(T-r,x)},
 \end{gather*}
 where $T$ is a tree that is obtained from a star with $k$ leaves, whose root is $r$, and the $i$th leaf is glued to the root of $T^{\sigma_i}_{G-u-\{u_1\dots u_{i-1}\},u_i}$. On the other hand this $T$ is isomorphic to  $T^\sigma_{G,u}$, since  any $\sigma$-stable path $P=(u,u_i,v_1,\dots,v_k)$   (specially if $1\le j<i$, then $u_j\notin\{v_1,\dots,v_k\}$) the path $P'=(u_i,v_1,\dots,v_k)$ is $\sigma_i$-stable. And for any $\sigma_i$-stable path $P'=(u_i,v_1,\dots,v_k)$  is a $P=(u,u_i,v_1,\dots,v_k)$ $\sigma$-stable path. So
 \[\frac{\oI(T-r,x)}{\oI(T,x)}=\frac{\oI(T^{\sigma}_{G,u}-\overline{u},x)}{\oI(T^{\sigma}_{G,u},x)}\]
\end{proof}

We would like to remark that Weitz's construction of the self-avoiding path tree is a special case of the previously defined stable-path tree of a deep decision. Let $\phi:E(G)\to \{1,\dots,m\}$ bijection, where $m=|E(G)|$. Then for a $(P,e)\in A_{G,u}$ let $\sigma(P,e)=\phi(e)$. Then $T^\sigma_{G,u}$ is the Weitz-tree.

\begin{Rem}
 Observe that if we have a deep decision for a connected graph, then we can perform the DFS-algorithm with respect to $\sigma$, in the following way. Whenever we arrive into the vertex $v$ along the path $P$ and there is an unvisited neighbor of $v$, then we will move to that unvisited vertex $w$ for which $\sigma(P,(v,w))$ is the smallest. 
 
 Formally, let us assume, that there is a given connected graph $G$, $u\in V(G)$ and a $\sigma$ deep decision from $u$. Then one can construct a spanning tree $F_{G,u,\sigma}$ (call as $\sigma$-DFS tree of $G$) as follows. Let $G_1,\dots,G_k$ be a the connected components of $G-u$, $u_i=\textrm{argmin}_{v\in V(G_i)\cap N_G(u)}(\sigma(u,(u,v)))$ for $1\le i\le k$ and the functions $\sigma_i:A_{G_i,u_i}\to \mathbb{R}$ are
 \begin{gather*}
  \sigma_i((u_i,v_1,\dots,v_k),e)=\sigma((u,u_i,v_1,\dots,v_k),e).
 \end{gather*}
 Then we gain $F_{G,u,\sigma}$ as we take the disjoint union of $F_{G_i,u_i,\sigma_i}$ for $1\le i\le k$ and we connect a new vertex called $u$ with $u_i$ for $1\le i\le k$.
\end{Rem}

By induction we can prove the following properties of a stable-path tree.
\begin{Pro}\label{prop:faszerk}
Let $G$ be a connected graph, $u\in V(G)$, $\sigma$ a deep decision, and let  $F$ be a $\sigma$-DFS tree. Denote by $\overline{F}$ the set of paths from $u$ in $F$ (they are $\sigma$-stable paths). Then
 \begin{enumerate}
  \item there exists a sequence $G_1, \dots, G_k$  of induced subgraphs of $G$, such that
    \[
      \oI(T^\sigma_{G,u},x)=\oI(G,x)\oI(G_1,x)\dots \oI(G_k,x),
    \]
  \item and
    \[
      \oI(G,x)=\frac{\oI(T^\sigma_{G,u})}{\oI(T^\sigma_{G,u}-\overline{F},x)}.
    \]
 \end{enumerate}
\end{Pro}
\begin{proof}
 We will prove the first part by induction on the number of vertices of $G$. The proof of the second part goes similarly. From the proof of the previous theorem (and with its notations) we know that
 \begin{gather*}
  \oI(T^\sigma_{G,u},x)=\frac{\oI(G,x)}{\oI(G-u,x)}\oI(T^\sigma_{G,u}-\overline{u},x)=\\
  =\frac{\oI(G,x)}{\oI(G-u,x)}\oI(T^{\sigma_1}_{G-u,u_1},x)\oI(T^{\sigma_2}_{G-\{u,u_1\},u_2},x)\dots \oI(T^{\sigma_d}_{G-\{u,u_1,\dots u_{d-1}\},u_d},x)=\\
  =\frac{\oI(G,x)}{\oI(G-u,x)}\prod_{i=1}^d\prod_{j=0}^{l_i}\oI(G^i_j,x),
 \end{gather*}
 where $G^i_0$ is the connected component of $G-\{u,u_1,\dots,u_{i-1}\}$, which contains $u_i$; and each $G^i_j$ is an induced subgraph of $G^i_0$. So each $G^i_j$ is an induced subgraph of $G$. 
 Let $\{H_1,\dots,H_t\}$ the set of connected components of $G-u$, and 
 \[
  I=\{\min_{u_i\in V(H_j)}(i)~|~1\le j \le t\}.
 \]
 By definition of $I$ we have that the set $\{G^i_0~|~i\in I\}$ is the set of connected components of $G-u$. This implies that the product $\prod_{i\in I}I(G^i_0,x)=I(G-u,x)$, therefore 
% Observe that $G^1_0=G-u$, so the product is divisible by $\oI(G-u,x)$, therefore
 \begin{eqnarray}\label{prec}
      \oI(T^\sigma_{G,u},x)=\oI(G,x)\prod_{i\in I'}\oI(G^i_0,x)\prod_{i=1}^d\prod_{j=1}^{l_i}\oI(G^i_j,x),
 \end{eqnarray}
 where $I'=\{1,\dots,d\}\setminus I$.
\end{proof}

\begin{Rem}
 Sometimes, it is useful to follow the induction to determine explicitly the multiplicites of the subgraphs occuring in the formula (\ref{prec}).
 %Sometimes, when we would like to determine the multiplicities of subgraphs occurring in the lemma, we have to carefully follow  the induction, as in formula (\ref{prec}).
\end{Rem}

\section{Applications of stable-path tree}\label{sec:ap}

In this section we will present various applications of the  following corollary of Proposition~\ref{prop:faszerk}:

\begin{Cor}\label{cor:poly}
 Let $G$ be a graph, $v\in V(G)$, and let $\sigma$ be a deep decision. If $G$ is a claw-free graph, then $I(T^\sigma_{G,u},x)$ is real-rooted. Moreover $I(G,x)$ divides $I(T^{\sigma}_{G,u},x)$.
\end{Cor}
\begin{proof}
 Assume that $G$ is a claw-free graph. Then by Proposition~\ref{prop:faszerk} we have a sequence of induced subgraphs $G_1,\dots,G_k$ of $G$, such that 
 \[
  I(T^\sigma_{G,u})=I(G,x)\prod_{i=1}^kI(G_i,x).
 \]
 Since each $G_i$ is an induced subgraph of a claw-free graph, therefore it is also claw-free. Then by the result of Chudnovsky and Seymour, Thm. 1.1. of \cite{sey}, we have that each polynomial $I(G_i,x)$ and the polynomial $I(G,x)$ are real-rooted, so their product is also real rooted.
\end{proof}

In this section will show some applications of this corollary. 
In all applications, of Corollary~\ref{cor:poly} the vertices of $G$ will be labelled by integers. This labeling will induce a total order on the vertices in the most natural way, the order of two vertices will be the order of their labels.  
% So it makes sense for any $G\in \{\widetilde{W}_n,\widetilde{H}_n,\widetilde{F}_n\}$ and $u_0\in V(G)$ to define a deep decision $\sigma_{G,u_0}$, which maps any pair $((u_0,\dots,u_k),(u_k,u_{k+1}))\in A_{G,u_0}$ into $u_{k+1}\in\mathbb{N}$. Since these particular $\sigma$'s are dependent only on $G$ and $u_0$, for simplicity we will write $T^\sigma_{G,u_0}$ instead of $T^{\sigma_{G,u_0}}_{G,u_0}$.
% 
% \begin{Rem}
%  https://cs.uwaterloo.ca/journals/JIS/VOL15/Benoumhani/benoumhani8.pdf
%  
%  http://www.sciencedirect.com/science/article/pii/S0195669810001125
%  
%  https://arxiv.org/pdf/math/0211036v1.pdf
%  
%  
% \end{Rem}

\subsection{Trees with real-rooted independence polynomial}
In this subsection we will show that some families of trees have real-rooted independence polynomials.
\begin{Def}
 Let us recall that,
 the \textsl{$n$-centipede $W_n$} is a graph such that we take a path on $n$ vertices and we hang 1 pendant edge from each vertex of it.
 
 The \textsl{$n$-caterpillar} $H_n$ is a graph such that we take a path on $n$ vertices and we hang 2 pendant edges from each vertex of it.
 
 The \textsl{Fibonacci tree} $F_0=K_1$ and $F_1=K_2$ with roots $r_0\in V(F_0)$ and $r_1\in V(F_1)$. Then for $n\ge 2$ the $n$th Fibonacci tree $F_n$ is obtained from  the disjoint union of $F_{n-1}$, $F_{n-2}$ and a new vertex, labeled as $r_n$, and connecting $r_n$ to the roots of $F_{n-1}$ and $F_{n-2}$. Define $r_n$ as the root of $F_n$.
\end{Def}

The proof of the real-rootedness of the independence polynomial of $W_n$ was in \cite{Zhu2007}, then a unified proof for $W_n$ and $H_n$ appeared in \cite{Wang2011}. The statement for $F_n$ was verified in \cite{Galvin2017} for $n\le 22$, and conjectured for arbitrary $n$. 
Our proofs will follow the following strategy: for each mentioned $T$ tree we will define a  claw-free graph $\widetilde{G}$ with integer labels, such that the stable-path tree of $\widetilde{G}$ from one of its vertex will be isomorphic to $T$.
% In the next proposition we will give a proof for all of them, but before that we need to define the following graphs.

\begin{Pro}\label{pro:centipede}
For any $n$, the independence polynomial of $W_n$ is real-rooted, hence log-concave and  unimodal.
\end{Pro}
\begin{proof}
  Let $\widetilde{W}_n$ be a graph (Fig. \ref{fig:tilde_Wn}), such that we take a path on $\{1,\dots,n\}$ and we attach a triangle to every $(2k+1)$th edge of the path. If $n$ is odd, then we attach a pendant edge to $n$. Also label all the new vertices by numbers bigger than $n$. 
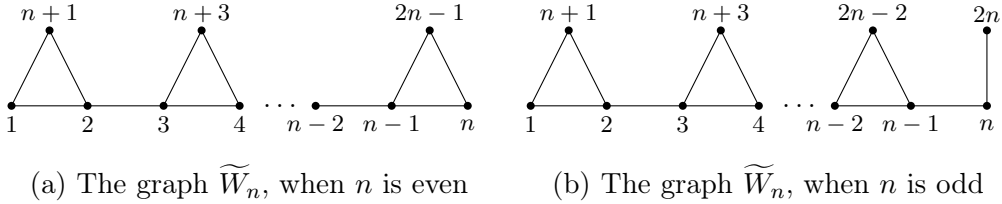
\begin{figure}[h]
\centering
\begin{subfigure}[b]{0.5\linewidth}
\begin{tikzpicture}[line cap=round,line join=round,>=triangle 45,x=1.0cm,y=1.0cm]
\clip(0.71,0.44) rectangle (7.63,2.4);
\draw (1,1)-- (2,1);
\draw (2,1)-- (3,1);
\draw (5,1)-- (6,1);
\draw (4.5,1) node {$$ \dots $$};
\draw (3,1)-- (4,1);
\draw (1,1)-- (1.5,2);
\draw (2,1)-- (1.5,2);
\draw (3.5,2)-- (3,1);
\draw (4,1)-- (3.5,2);
\draw (6,1)-- (7,1);
\draw (6.5,2)-- (7,1);
\draw (6.5,2)-- (6,1);
\begin{scriptsize}
\fill  (1,1) circle (1.5pt);
\draw (1,0.75) node {$1$};
\fill  (2,1) circle (1.5pt);
\draw (2,0.75) node {$2$};
\fill  (3,1) circle (1.5pt);
\draw (3,0.75) node {$3$};
\fill  (5,1) circle (1.5pt);
\draw (5,0.75) node {$n-2$};
\fill  (6,1) circle (1.5pt);
\draw (6,0.75) node {$n-1$};
\fill  (4,1) circle (1.5pt);
\draw (4,0.75) node {$4$};
\fill  (1.5,2) circle (1.5pt);
\draw (1.5,2.22) node {$n+1$};
\fill  (3.5,2) circle (1.5pt);
\draw (3.5,2.22) node {$n+3$};
\fill  (7,1) circle (1.5pt);
\draw (7,0.75) node {$n$};
\fill  (6.5,2) circle (1.5pt);
\draw (6.5,2.22) node {$2n-1$};
\end{scriptsize}
\end{tikzpicture}
\caption{The graph $\widetilde{W}_n$, when $n$ is even}
\end{subfigure}%
\begin{subfigure}[b]{0.5\linewidth}
\begin{tikzpicture}[line cap=round,line join=round,>=triangle 45,x=1.0cm,y=1.0cm]
\clip(0.71,0.44) rectangle (7.63,2.4);
\draw (1,1)-- (2,1);
\draw (2,1)-- (3,1);
\draw (5,1)-- (6,1);
\draw (4.5,1) node {$$ \dots $$};
\draw (3,1)-- (4,1);
\draw (1,1)-- (1.5,2);
\draw (2,1)-- (1.5,2);
\draw (3.5,2)-- (3,1);
\draw (4,1)-- (3.5,2);
\draw (6,1)-- (7,1);
\draw (5.5,2)-- (6,1);
\draw (5,1)-- (5.5,2);
\draw (7,1)-- (7,2);
\begin{scriptsize}
\fill  (1,1) circle (1.5pt);
\draw (1,0.75) node {$1$};
\fill  (2,1) circle (1.5pt);
\draw (2,0.75) node {$2$};
\fill  (3,1) circle (1.5pt);
\draw (3,0.75) node {$3$};
\fill  (5,1) circle (1.5pt);
\draw (5,0.75) node {$n-2$};
\fill  (6,1) circle (1.5pt);
\draw (6,0.75) node {$n-1$};
\fill  (4,1) circle (1.5pt);
\draw (4,0.75) node {$4$};
\fill  (1.5,2) circle (1.5pt);
\draw (1.5,2.22) node {$n+1$};
\fill  (3.5,2) circle (1.5pt);
\draw (3.5,2.22) node {$n+3$};
\fill  (7,1) circle (1.5pt);
\draw (7,0.75) node {$n$};
\fill  (5.5,2) circle (1.5pt);
\draw (5.5,2.22) node {$2n-2$};
\fill  (7,2) circle (1.5pt);
\draw (7,2.22) node {$2n$};
\end{scriptsize}
\end{tikzpicture}
\caption{The graph $\widetilde{W}_n$, when $n$ is odd} 
\end{subfigure}
\caption{The graph family $\widetilde{W}_n$}
\label{fig:tilde_Wn}
\end{figure}

  These graphs are claw-free, and 
  \begin{gather*}
   T^<_{\widetilde{W}_n,1}\cong W_n.
  \end{gather*}
 Therefore by Corollary~\ref{cor:poly} we have the desired statement.
\end{proof}

\begin{Pro}\label{pro:caterpillar}
For any $n$, the independence polynomial of $H_n$ are real-rooted, hence log-concave and  unimodal.
\end{Pro}
\begin{proof}
  Let $\widetilde{H}_n$ be a graph (Fig.~\ref{fig:tilde_Hn}), such that we take a path on $\{0,\dots,n+1\}$ and we attach a triangle to each edge, which is not the first or the last. Also label all the new vertices with numbers bigger than $n$.
  
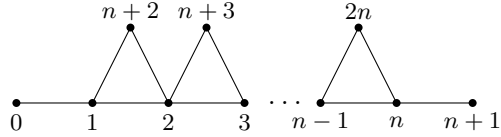
\begin{figure}[h!]
\begin{tikzpicture}[line cap=round,line join=round,>=triangle 45,x=1.0cm,y=1.0cm]
\clip(0.71,0.44) rectangle (7.63,2.4);
\draw (1,1)-- (2,1);
\draw (2,1)-- (3,1);
\draw (5,1)-- (6,1);
\draw (4.5,1) node {$$ \dots $$};
\draw (3,1)-- (4,1);
\draw (6,1)-- (7,1);
\draw (2.5,2)-- (2,1);
\draw (3,1)-- (2.5,2);
\draw (4,1)-- (3.5,2);
\draw (5,1)-- (5.5,2);
\draw (3.5,2)-- (3,1);
\draw (6,1)-- (5.5,2);
\begin{scriptsize}
\fill  (1,1) circle (1.5pt);
\draw (1,0.75) node {$0$};
\fill  (2,1) circle (1.5pt);
\draw (2,0.75) node {$1$};
\fill  (3,1) circle (1.5pt);
\draw (3,0.75) node {$2$};
\fill  (5,1) circle (1.5pt);
\draw (5,0.75) node {$n-1$};
\fill  (6,1) circle (1.5pt);
\draw (6,0.75) node {$n$};
\fill  (4,1) circle (1.5pt);
\draw (4,0.75) node {$3$};
\fill  (7,1) circle (1.5pt);
\draw (7,0.75) node {$n+1$};
\fill  (2.5,2) circle (1.5pt);
\draw (2.5,2.22) node {$n+2$};
\fill  (3.5,2) circle (1.5pt);
\draw (3.5,2.22) node {$n+3$};
\fill  (5.5,2) circle (1.5pt);
\draw (5.5,2.22) node {$2n$};
\end{scriptsize}
\end{tikzpicture}

 \caption{The graph $\widetilde{H}_n$}
 \label{fig:tilde_Hn}
\end{figure}

  These graphs are claw-free, and 
  \begin{gather*}
   T^<_{\widetilde{H}_n,0}\cong H_n.
  \end{gather*}
 Therefore by Corollary~\ref{cor:poly} we have the desired statement.
\end{proof}

\begin{Pro}
For any $n$, the independence polynomial of $F_n$ are real-rooted, hence log-concave and  unimodal.
\end{Pro}
\begin{proof}
  Let $\widetilde{F}_n$ be a graph (Fig.~\ref{fig:tilde_Fn}), such that we take the set $\{0,\dots,n-1\}$ and we connect $i$ and $j$ if $0<|i-j|\le 2$.
  
\begin{figure}[h]
\begin{tikzpicture}[line cap=round,line join=round,>=triangle 45,x=1.0cm,y=1.0cm]
\clip(0.71,0.44) rectangle (7.63,2.4);
\draw (1,1)-- (2,1);
\draw (2,1)-- (3,1);
\draw (5,1)-- (6,1);
\draw (4.5,1) node {$$ \dots $$};
\draw (3,1)-- (4,1);
\draw (6,1)-- (7,1);
\draw [shift={(2,1)}] plot[domain=0:pi,variable=\t]({1*1*cos(\t r)+0*1*sin(\t r)},{0*1*cos(\t r)+1*1*sin(\t r)});
\draw [shift={(3,1)}] plot[domain=0:pi,variable=\t]({1*1*cos(\t r)+0*1*sin(\t r)},{0*1*cos(\t r)+1*1*sin(\t r)});
\draw [shift={(6,1)}] plot[domain=0:pi,variable=\t]({1*1*cos(\t r)+0*1*sin(\t r)},{0*1*cos(\t r)+1*1*sin(\t r)});
\draw [shift={(4,1)}] plot[domain=1.57:pi,variable=\t]({1*1*cos(\t r)+0*1*sin(\t r)},{0*1*cos(\t r)+1*1*sin(\t r)});
\draw [shift={(5,1)}] plot[domain=0:1.57,variable=\t]({1*1*cos(\t r)+0*1*sin(\t r)},{0*1*cos(\t r)+1*1*sin(\t r)});
\begin{scriptsize}
\fill  (1,1) circle (1.5pt);
\draw (1,0.75) node {$0$};
\fill  (2,1) circle (1.5pt);
\draw (2,0.75) node {$1$};
\fill  (3,1) circle (1.5pt);
\draw (3,0.75) node {$2$};
\fill  (5,1) circle (1.5pt);
\draw (5,0.75) node {$n-3$};
\fill  (6,1) circle (1.5pt);
\draw (6,0.75) node {$n-2$};
\fill  (4,1) circle (1.5pt);
\draw (4,0.75) node {$3$};
\fill  (7,1) circle (1.5pt);
\draw (7,0.75) node {$n-1$};
\end{scriptsize}
\end{tikzpicture}

 \caption{The graph $\widetilde{F}_n$}
 \label{fig:tilde_Fn}
\end{figure}
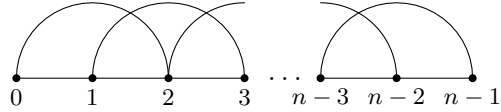

  These graphs are claw-free, and 
  \begin{gather*}
   T^<_{\widetilde{F}_n,0}\cong F_n.
  \end{gather*}
 Therefore by Corollary~\ref{cor:poly} we have the desired statement.
\end{proof}

% \begin{Def}\label{def:witness}
%  Let us define the following graph families:
%  \begin{itemize}
%   \item Let $\widetilde{W}_n$ be a graph, such that we take a path on $\{1,\dots,n\}$ and we attach a triangle to every $(2k+1)$th edge of the path. If $n$ is odd, then we attach a pendant edge to $n$. Also label all the new vertices by numbers bigger than $n$.
%   \item Let $\widetilde{H}_n$ be a graph, such that we take a path on $\{0,\dots,n+1\}$ and we attach a triangle to each edge, which is not the first or the last. Also label all the new vertices with numbers bigger than $n$.
%   \item Let $\widetilde{F}_n$ be a graph, such that we take the set $\{0,\dots,n-1\}$ and we connect $i$ and $j$ if $0<|i-j|\le 2$.
%  \end{itemize}
% \end{Def}
% 
% It is not hard to see that the graphs in Definition~\ref{def:witness} are claw-free graphs. 
% 
% \begin{Pro}
% For any $n$, the independence polynomial of $W_n$, $H_n$ and $F_n$ are real-rooted, hence log-concave and  unimodal.
% \end{Pro}
% \begin{proof}
%   As we pointed out the graphs $\widetilde{W}_n$, $\widetilde{H}_n$ and $\widetilde{F}_n$ are claw-free graphs. Also observe that 
%   \begin{gather*}
%    T^\sigma_{\widetilde{W}_n,1}\cong W_n,\\
%    T^\sigma_{\widetilde{H}_n,0}\cong H_n,\\
%    T^\sigma_{\widetilde{F}_n,0}\cong F_n.
%   \end{gather*}
%  Therefore by Corollary~\ref{cor:poly} we have the desired statement.
% \end{proof}

\begin{Rem}
 If someone carefully examine the formula (\ref{prec}), then one might get the following identities:
 \begin{gather*}
  I(W_{n},x)=I(\widetilde{W}_{n})(1+x)^{\lfloor n/2\rfloor},\\
  I(H_n,x)=I(\widetilde{H}_n)(1+x)^{n-2},\\
  I(F_n,x)=\prod_{k=0}^n I(\widetilde{F}_{k},x)^{f_{n-k}},  
 \end{gather*}
where $f_{0}=1$, $f_{1}=0$ and $f_{n}=f_{n-1}+f_{n-2}$ for $n\ge 1$.
\end{Rem}

\subsection{Some real-rooted graph families}
In this subsection we show another approach to verify real-rootedness of independence polynomials of some graphs. The idea is that for a graph $G$ we construct a  stable-path tree $T$, which is real-rooted. Then  by Corollary~\ref{cor:poly} we know that $I(G,x)$ divides $I(T,x)$, so it means that $I(G,x)$ is also real-rooted.
\begin{Def}
  Let us define the following graph families.
  
    The \textit{$n$th apple graph} $A_n$ is a graph (Fig.~\ref{fig:An}), such that we take a path on $\{1,\dots,n\}$, and we add the edge $(2,n)$.

    The \textit{$n$-sunlet} graph $N_n$ is a graph (Fig.~\ref{fig:Nn}), such that we take a cycle on $\{1,\dots, n\}$, and we attach a new vertex to each vertex of the cycle. Also label all the new vertices with numbers bigger than $n$. 

   Let $M_n$ be a graph (Fig.~\ref{fig:Mn}), such that we take a path on $\{1,\dots,n\}$, and we attach 2 triangles to any $2k+1$th edge of the path. If $n$ is odd, then we attach 2 pendant edges to $n$. For the new vertices choose different numbers greater than $n$ as labels.
%   Let $\widetilde{A}_n$ be a graph, such that we take a path on $\{1,\dots,n\}$, and add the edge $(2,4)$.
\end{Def}

\begin{figure}[h!]
\begin{subfigure}[b]{0.5\linewidth}
\begin{tikzpicture}[line cap=round,line join=round,>=triangle 45,x=1.0cm,y=1.0cm]
\clip(0.71,0.44) rectangle (7.63,2.4);
\draw (1,1)-- (2,1);
\draw (2,1)-- (3,1);
\draw (5,1)-- (6,1);
\draw (4.5,1) node {$$ \dots $$};
\draw (3,1)-- (4,1);
\draw (6,1)-- (7,1);
\draw [shift={(4.5,-3)}] plot[domain=1.01:2.13,variable=\t]({1*4.72*cos(\t r)+0*4.72*sin(\t r)},{0*4.72*cos(\t r)+1*4.72*sin(\t r)});
\begin{scriptsize}
\fill  (1,1) circle (1.5pt);
\draw (1,0.75) node {$1$};
\fill  (2,1) circle (1.5pt);
\draw (2,0.75) node {$2$};
\fill  (3,1) circle (1.5pt);
\draw (3,0.75) node {$3$};
\fill  (5,1) circle (1.5pt);
\draw (5,0.75) node {$n-2$};
\fill  (6,1) circle (1.5pt);
\draw (6,0.75) node {$n-1$};
\fill  (4,1) circle (1.5pt);
\draw (4,0.75) node {$4$};
\fill  (7,1) circle (1.5pt);
\draw (7,0.75) node {$n$};
\end{scriptsize}
\end{tikzpicture}

 \caption{The apple graph ($A_n$)}
 \label{fig:An}
\end{subfigure}%
\begin{subfigure}[b]{0.5\linewidth}
\begin{tikzpicture}[line cap=round,line join=round,>=triangle 45,x=1.0cm,y=1.0cm]
\clip(0.64,-0.17) rectangle (6.63,2.41);
\draw (1,1)-- (1,2);
\draw (1,1)-- (2,1);
\draw (2,1)-- (3,1);
\draw (2,1)-- (2,2);
\draw (3,1)-- (3,2);
\draw (5,1)-- (6,1);
\draw (6,1)-- (6,2);
\draw (5,1)-- (5,2);
\draw (4,1) node {$$ \dots $$};
\draw [shift={(3.5,5)}] plot[domain=4.15:5.27,variable=\t]({1*4.72*cos(\t r)+0*4.72*sin(\t r)},{0*4.72*cos(\t r)+1*4.72*sin(\t r)});
\begin{scriptsize}
\fill  (1,1) circle (1.5pt);
\fill  (2,1) circle (1.5pt);
\fill  (3,1) circle (1.5pt);
\fill  (5,1) circle (1.5pt);
\fill  (6,1) circle (1.5pt);
\fill  (1,2) circle (1.5pt);
\fill  (2,2) circle (1.5pt);
\fill  (3,2) circle (1.5pt);
\fill  (5,2) circle (1.5pt);
\fill  (6,2) circle (1.5pt);
\end{scriptsize}
\end{tikzpicture}

 \caption{The sunlet graph ($N_n$)}
 \label{fig:Nn}
\end{subfigure}
\caption{Some graph families}
\end{figure}
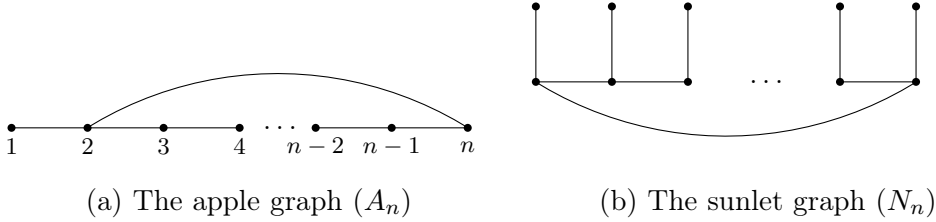

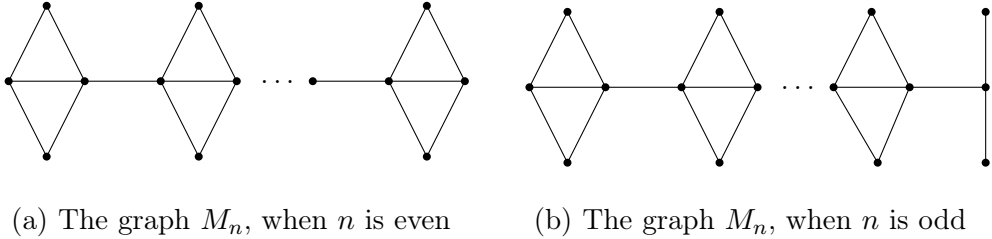
\begin{figure}[h]
\begin{subfigure}[b]{0.5\linewidth}
\begin{tikzpicture}[line cap=round,line join=round,>=triangle 45,x=1.0cm,y=1.0cm]
\clip(0.51,-0.43) rectangle (7.5,2.26);
\draw (1,1)-- (2,1);
\draw (2,1)-- (3,1);
\draw (5,1)-- (6,1);
\draw (4.5,1) node {$$ \dots $$};
\draw (3,1)-- (4,1);
\draw (1,1)-- (1.5,2);
\draw (2,1)-- (1.5,2);
\draw (3.5,2)-- (3,1);
\draw (4,1)-- (3.5,2);
\draw (6,1)-- (7,1);
\draw (6.5,2)-- (7,1);
\draw (6.5,2)-- (6,1);
\draw (1.5,0)-- (1,1);
\draw (1.5,0)-- (2,1);
\draw (3,1)-- (3.5,0);
\draw (3.5,0)-- (4,1);
\draw (6,1)-- (6.5,0);
\draw (6.5,0)-- (7,1);
\begin{scriptsize}
\fill  (1,1) circle (1.5pt);
\fill  (2,1) circle (1.5pt);
\fill  (3,1) circle (1.5pt);
\fill  (5,1) circle (1.5pt);
\fill  (6,1) circle (1.5pt);
\fill  (4,1) circle (1.5pt);
\fill  (1.5,2) circle (1.5pt);
\fill  (3.5,2) circle (1.5pt);
\fill  (7,1) circle (1.5pt);
\fill  (6.5,2) circle (1.5pt);
\fill  (1.5,0) circle (1.5pt);
\fill  (3.5,0) circle (1.5pt);
\fill  (6.5,0) circle (1.5pt);
\end{scriptsize}
\end{tikzpicture}

 \caption{The graph $M_n$, when $n$ is even}
\end{subfigure}%
\begin{subfigure}[b]{0.5\linewidth}
\begin{tikzpicture}[line cap=round,line join=round,>=triangle 45,x=1.0cm,y=1.0cm]
\clip(0.49,-0.35) rectangle (7.64,2.3);
\draw (1,1)-- (2,1);
\draw (2,1)-- (3,1);
\draw (5,1)-- (6,1);
\draw (4.5,1) node {$$ \dots $$};
\draw (3,1)-- (4,1);
\draw (1,1)-- (1.5,2);
\draw (2,1)-- (1.5,2);
\draw (3.5,2)-- (3,1);
\draw (4,1)-- (3.5,2);
\draw (6,1)-- (7,1);
\draw (5.5,2)-- (6,1);
\draw (5,1)-- (5.5,2);
\draw (7,1)-- (7,2);
\draw (1.5,0)-- (1,1);
\draw (2,1)-- (1.5,0);
\draw (3.5,0)-- (3,1);
\draw (4,1)-- (3.5,0);
\draw (5.58,0)-- (5,1);
\draw (5.58,0)-- (6,1);
\draw (7,0)-- (7,1);
\begin{scriptsize}
\fill  (1,1) circle (1.5pt);
\fill  (2,1) circle (1.5pt);
\fill  (3,1) circle (1.5pt);
\fill  (5,1) circle (1.5pt);
\fill  (6,1) circle (1.5pt);
\fill  (4,1) circle (1.5pt);
\fill  (1.5,2) circle (1.5pt);
\fill  (3.5,2) circle (1.5pt);
\fill  (7,1) circle (1.5pt);
\fill  (5.5,2) circle (1.5pt);
\fill  (7,2) circle (1.5pt);
\fill  (1.5,0) circle (1.5pt);
\fill  (3.5,0) circle (1.5pt);
\fill  (5.58,0) circle (1.5pt);
\fill  (7,0) circle (1.5pt);
\end{scriptsize}
\end{tikzpicture}

 \caption{The graph $M_n$, when $n$ is odd}
\end{subfigure}
\caption{The graph family $M_n$}
\label{fig:Mn}
\end{figure}

A proof for real-rootedness of the independence polynomial of $M_n$ and $N_n$ was given in \cite{Wang2011}. 

\begin{Pro}
For any $n$, the independence polynomial of $M_n$ is real-rooted, hence log-concave and  unimodal.
\end{Pro}
\begin{proof}
 By Proposition ~\ref{pro:caterpillar} we have that $H_n$ has real-rooted independence polynomial. However we can see that
 \[
  T^<_{M_n,1}\cong H_{n}.
 \]
 By Corollary~\ref{cor:poly} we know that $I(M_n,x)$ divides $I(H_{n},x)$, which implies, that $I(M_n,x)$ is real-rooted polynomial.
\end{proof}

\begin{Pro}
For any $n$, the independence polynomial of $N_n$ is real-rooted, hence log-concave and  unimodal.
\end{Pro}
\begin{proof}
 By Proposition ~\ref{pro:centipede} we have that $W_n$ has real-rooted independence polynomial. However we can see that
 \[
  T^<_{N_n,1}\cong W_{2n-1}.
 \]
 By Corollary~\ref{cor:poly} we know that $I(N_n,x)$ divides $I(W_{2n-1},x)$, which implies, that $I(N_n,x)$ is real-rooted polynomial.
\end{proof}

\begin{Pro}
For any $n\ge 4$, the independence polynomial of $A_n$ is real-rooted, hence log-concave and  unimodal.
\end{Pro}
\begin{proof}
 Let $\widetilde{A}_n$ be a graph (Fig.~\ref{fig:tilde_An}), such that we take a path on $\{1,\dots,n\}$, and add the edge $(2,4)$.
 \begin{figure}[h!]
 \begin{tikzpicture}[line cap=round,line join=round,>=triangle 45,x=1.0cm,y=1.0cm]
\clip(0.54,0.25) rectangle (8.6,2.7);
\draw (1,1.5)-- (2,1.5);
\draw (2,1.5)-- (3,2);
\draw (3,2)-- (3,1);
\draw (3,1)-- (2,1.5);
\draw (3,1)-- (4,1);
\draw (4,1)-- (5,1);
\draw (7,1)-- (8,1);
\draw (6,1) node {$$ \dots $$};
\begin{scriptsize}
\fill  (1,1.5) circle (1.5pt);
\draw (1,1.25) node {$1$};
\fill  (2,1.5) circle (1.5pt);
\draw (2,1.25) node {$2$};
\fill  (3,1) circle (1.5pt);
\draw (3,0.75) node {$4$};
\fill  (3,2) circle (1.5pt);
\draw (3,2.22) node {$3$};
\fill  (4,1) circle (1.5pt);
\draw (4,0.75) node {$5$};
\fill  (5,1) circle (1.5pt);
\draw (5,0.75) node {$6$};
\fill  (7,1) circle (1.5pt);
\draw (7,0.75) node {$n-1$};
\fill  (8,1) circle (1.5pt);
\draw (8,0.75) node {$n$};
\end{scriptsize}
\end{tikzpicture}

  \caption{The graph $\widetilde{A}_n$}
  \label{fig:tilde_An}
 \end{figure}
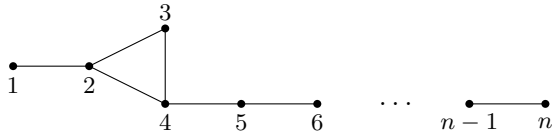

 Since $\widetilde{A}_n$ is a claw-free graph, so for any $n\ge 4$ we have that $T_{\widetilde{A}_n,1}$ has a real-rooted independence polynomial. However we can see that
 \[
  T^<_{\widetilde{A}_n,1}\cong T^<_{A_n,1},
 \]
 which means that $I(T^<_{A_n,1},x)$ is real-rooted. By Corollary~\ref{cor:poly} we know that $I(A_n,x)$ divides $I(T^<_{A_n,1},x)$, which implies, that $I(A_n,x)$ is real-rooted polynomial.
%  
%  The case of $N_n$ and $M_n$ are similar, since
%  \begin{gather*}
%   T^\sigma_{\widetilde{N}_n,1}\cong W_{2n-1},\\
%   T^\sigma_{M_n,1}\cong H_{n}
%  \end{gather*}
%  and we know that $I(W_{2n-1},x)$ and $I(H_n,x)$ are real-rooted polynomials.
\end{proof}

\section{Final remarks}

We would like to  remark, that this method can be also capable of proving the real-rootedness of the independence polynomial of the ladder graph (Thm. 5.1. of \cite{Zhu2015a}), the polyphenyl ortho-chain ($\bar{O}_n$ of \cite{alikhani2011}), $k$-ary analogue of the Fibonacci tree (Remark of \cite{Wagner2007}). 

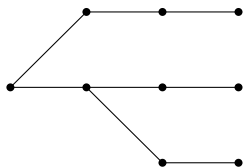
\begin{figure}
  \begin{tikzpicture}[line cap=round,line join=round,>=triangle 45,x=1.0cm,y=1.0cm]
  \clip(0.73,0.72) rectangle (4.4,3.3);
  \draw (1,2)-- (2,2);
  \draw (2,2)-- (3,2);
  \draw (3,2)-- (4,2);
  \draw (1,2)-- (2,3);
  \draw (2,3)-- (3,3);
  \draw (3,3)-- (4,3);
  \draw (2,2)-- (3,1);
  \draw (3,1)-- (4,1);
  \begin{scriptsize}
  \fill  (1,2) circle (1.5pt);
  \fill  (2,2) circle (1.5pt);
  \fill  (3,2) circle (1.5pt);
  \fill  (4,2) circle (1.5pt);
  \fill  (2,3) circle (1.5pt);
  \fill  (3,3) circle (1.5pt);
  \fill  (4,3) circle (1.5pt);
  \fill  (3,1) circle (1.5pt);
  \fill  (4,1) circle (1.5pt);
  \end{scriptsize}
  \end{tikzpicture}
  \caption{A tree $T$ with real-rooted independence polynomial, which is not a stable-path tree of any non-tree graph}
  \label{ellentree}
\end{figure}
One might ask that  it is true that any tree with real-rooted independence polynomial is a stable path tree of a non-tree graph $G$. The answer is no, as the following example shows:

Let $T$ be a tree on 9 vertices as on the Figure~\ref{ellentree} and assume that there exists a graph $G$, a deep decision $\sigma$ and a vertex $u\in V(G)$, such that $T=T^\sigma_{G,u}$. Then the independence polynomial of $T$ is 
\begin{gather*}
  I(T,x)=(1+3x+x^2)(1+5x+6x^2+x^3)+x(1+2x)^3=\\
  (1+x)(1+8x+20x^2+16x^3+x^4),
\end{gather*}
where the factors are real-rooted and irreducible polynomials in $\mathbb{Q}[x]$. By Proposition~\ref{prop:faszerk} we have that $I(G,x)$ divides $I(T,x)$, and clearly $G$ cannot be $K_1$ or the empty graph, therefore $I(G,x)$ should be $1+8x+20x^2+16x^3+x^4$. However it can be proved, that there is no such a graph $G$.
% 
% \section{Concluding remarks}
% In the paper we showed some factorizations of certain ''recursively'' defined trees, and by it we gained a method to construct trees with real-rooted independence polynomials.
% 
% In the proof of the main theorem 

% \vspace{10pt}\noindent \textbf{Acknowledgments.}
%  I would like to express my sincere gratitude to P\'eter Csikv\'ari for introducing me a set of problems which led to this paper. I would also like to thank for his gentle guidance and great effort which helped my work run smoothly. These results would not have been achieved without his supervision.
  
\bibliography{hivatkozat}
\bibliographystyle{plain}

\end{document}